\documentclass{amsart}%
\usepackage{amssymb}
\usepackage{amsfonts}
\usepackage{amsmath}
\usepackage{graphicx}%
\newtheorem{theorem}{Theorem}
\theoremstyle{plain}

\newtheorem{example}{Example}

\newtheorem{lemma}{Lemma}

\numberwithin{equation}{section}
\begin{document}
\title[Sierpinski--Hamming?]{Can the Sierpinski Graph be Embedded in the Hamming Graph?}
\author{Lawrence Hueston Harper}
\address{Department of Mathematics\\
University of California-Riverside}
\email{harper@math.ucr.edu}
\date{November 8, 20015}

\begin{abstract}
The (generalized \& expanded) Sierpinsi graph, $S(n,m)$, and the Hamming
graph, $K_{m}^{n}$, have the same set of vertices, $\left\{
0,1,...,m-1\right\}  ^{n}$. The edges of both are (unordered) pairs of
vertices. Each set of edges is defined by a different property so that neither
contains the other. We ask if there is a subgraph of $K_{m}^{n}$ isomorphic to
$S(n,m)$ and show that the answer is yes. The recursively defined embedding
map, $\varphi:S(n,m)\rightarrow K_{m}^{n}$, leads to a number of variations
and ramifications. Among them is a simple algebraic formula for the solution
of the Tower of \ Hanoi puzzle.

\end{abstract}
\dedicatory{In memory of my father,\\\textbf{Hueston Maxwell Harper}\\1910-2005\\Dad was thrown out of his high school algebra class for asking questions such
as, "What is algebra?" and "What do 'x' and 'y' mean?" He never went back. But
he loved manipulable puzzles. I was about six when he showed me how to
untangle the Chinese Rings (a physical manifestation of the Sierpinski graph).}\maketitle

\section{ Introduction}

\subsection{Background}

The standard architecture (connection graph) for\ a multiprocessor computer is
the $n$-dimensional cube with $n=6$ or $7$ (so the number of processors is 64
or 128). Recently, the "Sierpinski gasket pyramid graph" has been proposed by
William, Rajasingh, Rajan \& Shantakumari \cite{W-R-R-S} as an alternative (to
the $n$-cube) . After determining the elementary graph-theoretic properties of
the Sierpinski gasket pyramid graph, W-R-R\&S proposed studying its "message
routing and broadcasting" properties. This is the first of several papers
following up on that suggestion.

Our goal with this paper has been to learn more about the structure of
Sierpinski graphs. Our goal for the next paper is to solve the
edge-isoperimetric problem ($EIP$) on the (generalized \& expanded) Sierpinski
graph, $S(n,m)$. $S(n,m)$ has the same vertex set as $K_{m}^{n}$, the Hamming
graph (see below for definitions), but fewer edges, so it was natural to ask
if $S(n,m)$ could be embedded in $K_{m}^{n}$. If so, it would provide a
context for $S(n,m)$ which could lead to insight. The structure of $K_{m}^{n}$
is relatively transparent, its $EIP$ has been solved (see \cite{Har04}, p.
112) and it is closely related to the $n$-cube (the Sierpinski gasket pyramid
graph's competition for multiprocessor architecture). Embedding $S(n,m)$ into
$K_{m}^{n}$ turned out to be an interesting and challenging problem in its own
right and its solution has unexpected ramifications.

\subsection{Definitions}

\subsubsection{Graphs}

An \textit{ordinary} \textit{graph, }$G=\left(  V,E\right)  $ consists of a
set $V$, of \textit{vertices} and a set $E\subseteq\binom{V}{2}$, of pairs of
vertices called \textit{edges}.

\begin{example}
\bigskip$K_{m}$, \textit{the complete graph on }$m$ vertices has $V_{K_{m}%
}=\left\{  0,1,2,...,m-1\right\}  $ and $E_{K_{m}}=\binom{V_{K_{m}}}{2} $.
\end{example}

\begin{example}
The (disjunctive) product, $K_{m}\times K_{m}\times...\times K_{m}=K_{m}^{n} $
is called the Hamming graph. Two vertices ($n$-tuples of vertices of $K_{m} $)
have an edge between them if they differ in exactly one coordinate
(\textit{i.e. }are at Hamming distance $1$).
\end{example}

\subsubsection{The \textit{(Generalized \& Expanded) Sierpinski Graph,
}$S(n,m)$}

$S(n,m)$, $n\geq1$, $m\geq2$ was defined in 1944 by Scorer, Grundy and Smith
\cite{S-G-S}: $V_{S(n,m)}=\left\{  0,1,...,m-1\right\}  ^{n}$. For $\left\{
u,v\right\}  \in\binom{V_{S(n,m)}}{2}$, $\left\{  u,v\right\}  \in E_{S(n,m)}$
iff $\exists h\in\left\{  1,2,...,n\right\}  $ such that following 3
conditions hold:

\begin{enumerate}
\item $u_{i}=v_{i}$ for $i=1,2,...h-1$;

\item $u_{h}\neq v_{h}$; and

\item $u_{j}=v_{h}$ and $v_{j}=u_{h}$ for $j=h+1,...,n$.
\end{enumerate}

The motivation for defining $S(n,m)$ was that $S(n,3)$ is the graph of the
$3$-peg Tower of Hanoi puzzle with $n$ disks \cite{S-G-S}. Scorer, Grundy and
Smith also pointed out that the graph defined by topologist Wac\l sav
Sierpinski in 1915, known as the Sierpinski gasket graph, is a contraction of
$S(n,3)$ where every edge of $S(n,3)$ not contained in a triangle ($K_{3}$) is
contracted to a vertex. The Sierpinski gasket pyramid (aka the Sierpinski
sponge), is a $3$-dimensional analog of the Sierpinski gasket, and its graph
is a similar contraction of $S(n,4)$. Jakovac \cite{Jak} generalized the
construction to $S[n,m]$, the contraction of $S(n,m)$ in which every edge of
$S(n,m)$ not contained in a $K_{3}$ is contracted to a vertex. He worked out
some properties of $S[n,m]$ such as hamiltonicity and chromatic number
($\chi\left(  S[n,m]\right)  =m$).

It seems apparent that the nomenclature of these structures is a muddle. First
came the Sierpinski gasket, a topological curiosity, and then the family of
graphs ($S\left[  n,3\right]  $) that comprise the boundary of the "gasket" at
the $n^{th}$ stage of its construction. The "gasket", a 2-dimensional
structure, was generalized to $m$ dimensions by replacing the equilateral
triangle that Sierpinski started with, by an $m$-simplex. If $S[n,3]$ is the
Sierpinski (gasket) graph, then $S[n,m]$ should be called the "generalized
Sierpinski graph". Scorer, Grundy \& Smith noted that $S\left[  n,3\right]  $
could be obtained from the Tower of Hanoi graph, $S\left(  n,3\right)  $, by
contracting certain edges to a single vertex. Conversely, $S\left(
n,3\right)  $ is derived from $S\left[  n,3\right]  $ by expanding certain
vertices into an edge (and two vertices) and $S(n,m)$ is a generalization of
that construction. For many purposes $S(n,m)$ is easier to deal with (than
$S[n,m]$) because its vertices and edges are easily defined. So we propose
calling $S(n,m)$ the "generalized \& expanded Sierpinski graph" or "Sierpinski
graph" for short (when the meaning is clear from context).

\section{Structure of $S(n,m)$}

\subsection{The Basics}

$\left\vert V_{S(n,m)}\right\vert =m^{n}$. All $v\in V_{S(n,m)}$ have $m-1$
"interior" neighbors. These are the $n$-tuples that agree with $v$ in all
coordinates except the $n^{th}$ (the case $h=n$ in the definition of edges in
$S(n,m)$). If $v\neq i^{n}$ then $v$ has one other ("exterior") neighbor: If
$v\neq i^{n}=\left(  i,i,...,i\right)  $, then $\exists h,1\leq h<n$, such
that $v_{h}\neq v_{h+1}=v_{h+2}=...=v_{n}$ and by definition the exterior
neighbor of $v$ is $u=v_{1}v_{2}...v_{h-1}v_{h+1}v_{h}v_{h}...v_{h}$ (note
that this relationship between $u$ and $v$ is symmetric). Thus $i^{n}$, with
$i=0,1,2,...,m-1$, has degree $m-1$ and every other vertex has degree $m$.
Summing the degrees of all vertices we get $m\left(  m-1\right)  +\left(
m^{n}-m\right)  m=$ $m^{n+1}-m$. Since each edge is incident to two vertices,
$\left\vert E_{S(n,m)}\right\vert =\left(  m^{n+1}-m\right)  /2$.

The vertices that agree in all except the last coordinate induce a complete
subgraph, $K_{m}$. There are $m^{n-1}$ such $K_{m}$'$s$, nonoverlapping and
containing all the vertices of $S(n,m)$. This constitutes a $K_{m}%
$-decomposition of $S(n,m)$. Since any vertex is incident to at most one
exterior edge, any triangle ($K_{3}$) must contain at least two internal
edges. But then the third edge would also be internal to the same $K_{m}$, so
our $K_{m}$-decomposition is unique. The vertices $i^{n}$, for
$i=0,1,2,...,m-1$, are called \textit{corner} vertices of $S(n,m)$.

\begin{example}
$S(n,2)$ is just a path of length $2^{n}$. Its endpoints are the two corner
vertices, $0^{n}$ \& $1^{n}$. Every other vertex is incident to two edges, one
coming from (the direction of) $0^{n}$ and the other going toward $1^{n}$.
That it is a single path (and not a path plus disjoint cycles) is shown by the
observation that the two vertices, $u,v$, connected by the edge $\left\{
u,v\right\}  $ are binary representations of two consecutive integers. One of
them, say $u$, has $u_{h}=0$ and the other $v_{h}=1$. The last $n-h+1$
components of $u$ contribute $2^{n-h-1}+2^{n-h-2}+...+1=2^{n-h}-1$ to its
binary representation whereas that of $v$ contributes $2^{n-h}$.
\end{example}

Figure 1 shows diagrams of $S(n,2)$ for $n=1,2,3$ with coordinates.
\[%
{\parbox[b]{2.8919in}{\begin{center}
\includegraphics[
trim=2.037533in 5.776840in 1.600797in 3.493177in,
height=1.0533in,
width=2.8919in
]%
{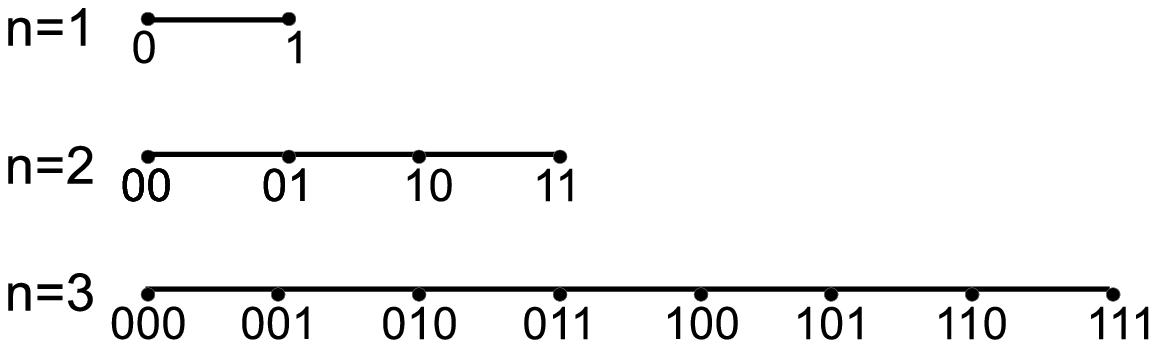}%
\\
Figure 1-$S(n,2)$ for $n=1,2,3$
\end{center}}}%
\]
Figure 2 shows diagrams of $S(n,2)$ for $n=1,2,3$ with coordinates.
\[%
{\parbox[b]{2.9914in}{\begin{center}
\includegraphics[
trim=1.200598in 2.406925in 1.367135in 1.205113in,
height=3.7239in,
width=2.9914in
]%
{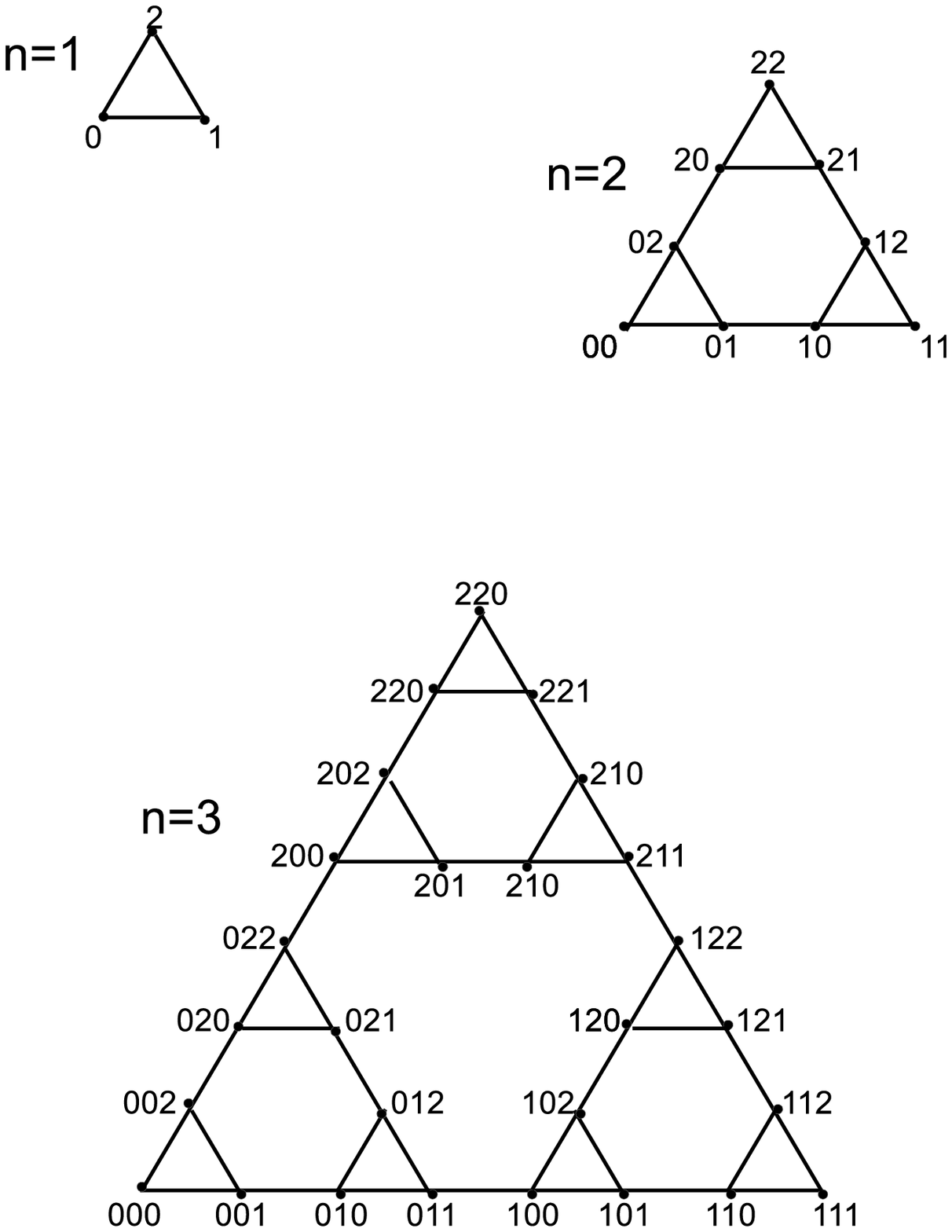}%
\\
Figure 2-$S(n,3)$ for $n=1,2,3$
\end{center}}}%
\]

\subsection{Symmetry of $S(n,m)$}

\begin{theorem}
(Theorem 4.14 of \cite{H-K-M-P}) The symmetry group of $S(n,m)$ is
$\mathcal{S}_{m}$, the symmetric group on $m$ generators.

\begin{proof}
Let $\pi$ be any permutation of $\left\{  0,1,2,...,m-1\right\}  $, and let it
act on the vertices of $S(n,m)$ in the obvious way,
\[
\sigma_{\pi}(v)=\sigma_{\pi}\left(  v_{1},v_{2},...,v_{n}\right)  =\left(
\pi\left(  v_{1}\right)  ,\pi\left(  v_{2}\right)  ,...,\pi\left(
v_{n}\right)  \right)  .
\]
If $\left\{  u,v\right\}  \in E_{S(n,m)}$ then
\begin{align*}
\sigma_{\pi}(u)  &  =\sigma_{\pi}\left(  u_{1},u_{2},...,u_{n}\right) \\
&  =\sigma_{\pi}\left(  u_{1},u_{2},...,u_{h-1},u_{h},u_{h+1},u_{h+1}%
,...,u_{h+1}\right)  \text{, where }u_{h}\neq u_{h+1}\text{,}\\
&  =\left(  \pi\left(  u_{1}\right)  ,\pi\left(  u_{2}\right)  ,...,\pi\left(
u_{h-1}\right)  ,\pi\left(  u_{h}\right)  ,\pi\left(  u_{h+1}\right)
,\pi\left(  u_{h+1}\right)  ,...,\pi\left(  u_{h+1}\right)  \right)
\end{align*}
and%
\begin{align*}
\sigma_{\pi}(v)  &  =\sigma_{\pi}\left(  v_{1},v_{2},...,v_{n}\right) \\
&  =\sigma_{\pi}\left(  v_{1},v_{2},...,v_{h-1},v_{h},v_{h+1},v_{h+1}%
,...,v_{h+1}\right) \\
&  =\sigma_{\pi}\left(  u_{1},u_{2},...,u_{h-1},u_{h+1},u_{h},u_{h}%
,...,u_{h}\right) \\
&  =\left(  \pi\left(  u_{1}\right)  ,\pi\left(  u_{2}\right)  ,...,\pi\left(
u_{h-1}\right)  ,\pi\left(  u_{h+1}\right)  ,\pi\left(  u_{h}\right)
,\pi\left(  u_{h}\right)  ,...,\pi\left(  u_{h}\right)  \right)  \text{.
\ \ \ }%
\end{align*}
So $\left\{  \sigma_{\pi}\left(  u\right)  ,\sigma_{\pi}\left(  v\right)
\right\}  \in E_{S(n,m)}$ and $\sigma_{\pi}$ is a symmetry of $S(n,m)$.

Conversely, suppose $\sigma$ is a symmetry of $S(n,m)$. Any symmetry, $\sigma
$, of $S(n,m)$ must take corners to corners so $\sigma\left(  i^{n}\right)
=j^{n}$ for some $j\in\left\{  0,1,...,m-1\right\}  $. We then define a
permutation on $\left\{  0,1,...,m-1\right\}  $ by $\pi_{\sigma}\left(
i\right)  =j$. We claim that when this action is extended as above, we get
$\sigma$ back again: By induction on $h$ (in the definition of external
edges). For $h=0,\sigma\left(  i^{n}\right)  =j^{n}=\pi_{\sigma}\left(
i\right)  ^{n}$. For $h=1$, if $\sigma\left(  ij^{n-1}\right)  =\left(
k_{1},k_{2},...,k_{n}\right)  $ then since the subgraph $\left\{  i\right\}
\times S(n-1,m)$ contains $i^{n}$, its image under $\sigma$ must contain
$\pi_{\sigma}\left(  i\right)  ^{n}$ and be $\left\{  \pi_{\sigma}\left(
i\right)  \right\}  \times S(n-1,m)$. Therefore $k_{1}=\pi_{\sigma}\left(
i\right)  $. $\sigma\left(  ij^{n-1}\right)  $ must be a corner vertex of
$\left\{  \pi_{\sigma}\left(  i\right)  \right\}  \times S(n-1,m)$ which means
that $\left(  k_{2},...,k_{n}\right)  =k^{n-1}$ for some $k.$ Since
$ij^{n-1}\in\left\{  i\right\}  \times S(n-1,m)$ and $ji^{n-1}\in\left\{
j\right\}  \times S(n-1,m)$ are neighbors, $\sigma\left(  ij^{n-1}\right)
\in\left\{  \pi_{\sigma}\left(  i\right)  \right\}  \times S(n-1,m)$ and
$\sigma\left(  ji^{n-1}\right)  \in\left\{  \pi_{\sigma}\left(  j\right)
\right\}  \times S(n-1,m)$ are also neighbors, so $k=\pi\left(  j\right)  $
and $\sigma\left(  ij^{n-1}\right)  =\pi_{\sigma}\left(  i\right)  \pi
_{\sigma}\left(  \ j\right)  ^{n-1}$. The pattern for higher $h$ is the same,
so $\sigma\left(  v_{1}v_{2}...v_{n}\right)  =\pi_{\sigma}\left(
v_{1}\right)  \pi_{\sigma}\left(  v_{2}\right)  ...\pi_{\sigma}\left(
v_{n}\right)  $ and the symmetry induced by $\pi_{\sigma}$ is $\sigma$.
\end{proof}
\end{theorem}

\section{Embedding $S(n,m)$ in $K_{m}^{n}$}

The vertex-sets of $S(n,m)$ and $K_{m}^{n}$ are the same, $\left\{
0,1,...,m-1\right\}  ^{n}$ . However the edge-sets are different:
$E_{K_{m}^{n}}\nsubseteqq E_{S(n,m)}$because
\[
\left\vert E_{K_{m}^{n}}\right\vert =n\left(  m-1\right)  m^{n}/2>\left(
m^{n+1}-m\right)  /2=\left\vert E_{S(n,m)}\right\vert .
\]
Also $\ E_{S(n,m)}\nsubseteqq E_{K_{m}^{n}}$ since the exterior edges of
$S(n,m)$ differ in two or more coordinates. The question of whether $S(n,m)$
can be embedded in $K_{m}^{n}$ is a special case of the subgraph isomorphism
problem, well-known for it complexity (it is NP-complete \cite{G-J}). However,
we now show that such an embedding exists. First we consider a couple of
special cases.

\subsubsection{The Case $m=2$}

We showed in Example 1 that $S(n,2)$ is a path from $0^{n}$ to $1^{n}$. So the
question of embedding of $S(n,2)$ into $K_{2}^{n}=Q_{n}$, the graph of the
$n$-cube, is the graph-theoretic classic, "Does $Q_{n}$ contain a Hamiltonian
path (a path passing through each vertex exactly once). The answer is yes
(there are many, the Gray code being one (see Wikipedia on "Gray code")).

\subsubsection{The Case $n=2$ (Giving $S(n,m)$ a twist)}

Define $\widetilde{S}(n,m)$ as the graph with
\[
V_{\widetilde{S}(n,m)}=\left\{  0,1,...,m-1\right\}  ^{n}%
\]
and
\[
E_{\widetilde{S}(n,m)}=%
{\displaystyle\bigcup\limits_{h=1}^{n}}
\left\{  \left\{  v^{\left(  h-1\right)  }ik^{n-h},v^{\left(  h-1\right)
}jk^{n-h}\right\}  :i\neq j\text{ \& }i+j=k(\operatorname{mod}m)\right\}
\]
where $v^{\left(  h\right)  }\in\left\{  0,1,...,m-1\right\}  ^{h}$. Note that
$\left\vert V_{\widetilde{S}(n,m)}\right\vert =\left\vert V_{S(n,m)}%
\right\vert $ and $\left\vert E_{\widetilde{S}(n,m)}\right\vert =\left\vert
E_{S(n,m)}\right\vert $. Also note that $\widetilde{S}(n,m)$ is a subgraph of
$K_{m}^{n}$ (since $V_{\widetilde{S}(n,m)}=V_{K_{m}^{n}}$ and the pairs of
vertices that comprise an edge of $\widetilde{S}(n,m)$ differ in exactly one
coordinate, making it an edge of $K_{m}^{n}$). However, this does not prove
that $S(n,m)$ is embeddable as a subgraph of $K_{m}^{n}$.

\begin{lemma}
$S(2,m)\cong\widetilde{S}(2,m)$.
\end{lemma}

In the next section we will prove a more general theorem. It would seem
natural to extend Lemma 1 to show that $\widetilde{S}(n,m)\cong S(n,m)$. After
many attempts to do so, we ran up against the following counterexample:
Consider the vertex $\left(  0,1,1\right)  $ in $\widetilde{S}(3,3)$. It
shares an edge with the vertices $\left(  0,1,0\right)  $ $\&$ $\left(
0,1,2\right)  $ ($h=3$) as well as the vertices $\left(  0,0,1\right)  $
($h=2$) and $\left(  1,1,1\right)  $ ($h=1$). So $\left(  0,1,1\right)  $ has
degree $4$ in $\widetilde{S}(3,3)$, but the maximum degree of any vertex in
$S(3,3)$ is $3$.

\subsubsection{A Recursive Definition of $S(n,m)$}

The Sierpinski graph, $S(n,m)$, has been defined analytically in Section
1.2.2. However, $S(n,m)$ may also be characterized recursively: $S(1,m)=K_{m}
$ and given $S(n-1,m)$ for $n-1\geq1$, we construct $S(n,m)$ by making $m$
copies, $\left\{  i\right\}  \times S(n-1,m)$, $0\leq i<m$, and throwing in
all edges of the form $\left\{  ij^{n-1},ji^{n-1}\right\}  $ such that $0\leq
i,j<m$ and $i\neq j$. Note that every copy of $S(n-1,m)$ is connected to every
other copy and the corners of $S(n,m)$ are the constant vertices $i^{n},$
$0\leq i<m$.

\subsection{The General Case (Giving $S(n,m)$ another Twist)}

The Scorer-Grundy-Smith definition of the (generalized and extended)
Sierpinski graph may be regarded as a representation or coordinatization of
the abstract graph. Our question of whether $S(n,m)$ may be embedded in
$K_{m}^{n\text{ \ }}$is equivalent to asking whether there is another
definition of that abstract graph; a definition in which the set of vertices,
$\left\{  0,1,...,m-1\right\}  ^{n}$, is the same, but the pairs of vertices,
$\left\{  u,v\right\}  $ that constitute an edge differ in exactly one
coordinate. An obvious candidate was $\widetilde{S}(n,m)$. That worked for
$n=2$ (Lemma 1) but failed for $n=3$.

\begin{theorem}
There exists another coordinatization of $S(n,m)$ (call it $\widetilde
{\widetilde{S}}(n,m)$) that is a subgraph of $K_{m}^{n}$. I.e. the edges of
$\widetilde{\widetilde{S}}(n,m)$ consist of pairs of vertices that differ in
exactly one coordinate.

\begin{proof}
By induction on $n$: It is trivial for $n=1$ ( $\widetilde{\widetilde{S}%
}(1,m)=K_{m}=S(1,m)$). Assume that the theorem is true for $n-1\geq1$.
Implicit in the statement of the theorem is the existence of a function,
$\varphi^{\left(  n\right)  }:\left\{  0,1,...,m-1\right\}  ^{n}%
\rightarrow\left\{  0,1,...,m-1\right\}  ^{n}$ that takes each vertex of
\ $S(n,m)$ to its coordinates in $\widetilde{\widetilde{S}}(n,m)$. We write
this as $\varphi^{\left(  n\right)  }\left(  v_{1}v_{2}...v_{n}\right)
=\left(  \widetilde{\widetilde{v}}_{1},\widetilde{\widetilde{v}}%
_{2},...,\widetilde{\widetilde{v}}_{n}\right)  $. We have represented $S(n,m)$
recursively as
\[
\left\{  0,1,...,m-1\right\}  \times S(n-1,m)
\]
(the disjoint union of $m$ copies of $S(n-1,m)$), along with the external
edges,
\[
\left\{  ij^{n-1},ji^{n-1}\right\}  ,i,j\in\left\{  0,1,...,m-1\right\}
\&i\neq j.
\]
By the inductive hypothesis we have $\varphi^{(n-1)}:S(n-1,m)\rightarrow$
$\widetilde{\widetilde{S}}(n-1,m)$ taking $v\in V_{S(n-1,m)}$ to its
coordinates in $\widetilde{\widetilde{S}}(n-1,m)$. For each $i\in\left\{
0,1,...,m-1\right\}  $ we define%
\[
\varphi_{i}^{\left(  n-1\right)  }:\left\{  i\right\}  \times
S(n-1,m)\rightarrow\left\{  i\right\}  \times\text{ }\widetilde{\widetilde{S}%
}(n-1,m)
\]
by
\[
\varphi_{i}^{\left(  n-1\right)  }\left(  ij^{n-1}\right)  =\left(
i,\varphi^{\left(  n-1\right)  }\left(  \left(  i+j\right)
(\operatorname{mod}m\right)  )^{n-1}\right)  \text{.}%
\]
The $ij^{n-1}$ are only the corner vertices of $\left\{  i\right\}  \times
S(n-1,m)$, but as we showed in Theorem 1, any permutation of the underlying
set, $\left\{  0,1,...,m-1\right\}  $ (and here we are applying the
permutation $j\rightarrow\left(  i+j\right)  \operatorname{mod}m$), extends to
a unique symmetry of $S(n-1,m)$. That extension constitutes $\varphi
_{i}^{\left(  n-1\right)  }$. Now we define
\[
\varphi^{(n)}=%
{\displaystyle\bigoplus\limits_{i=0}^{m-1}}
\varphi_{i}^{\left(  n-1\right)  }\text{,}%
\]
the disjoint union of the component maps, $\varphi_{i}^{\left(  n-1\right)  }%
$, $0\leq i\leq m-1$. We claim then that $\varphi^{(n)}:$ $S(n,m)\rightarrow$
$\widetilde{\widetilde{S}}(n,m)$ is the isomorphism we have been seeking:
Edges internal to $\left\{  i\right\}  \times S(n-1,m)$ are preserved because
$\varphi_{i}^{\left(  n-1\right)  }:\left\{  i\right\}  \times
S(n-1,m)\rightarrow\left\{  i\right\}  \times$ $\widetilde{\widetilde{S}%
}(n-1,m)$ is the composition of two isomorphisms $j\rightarrow\left(
i+j\right)  (\operatorname{mod}m)$ and $\varphi^{(n-1)}$. \ By induction on
$n$, the vertices incident to an edge internal to $\left\{  i\right\}  \times$
$\widetilde{\widetilde{S}}(n-1,m)$ differ in exactly one coordinate. 

The external edges, $\left\{  ij^{n-1},ji^{n-1}\right\}  ,$ are preserved
because%
\begin{align*}
&  \varphi^{(n)}\left(  \left\{  ij^{n-1},ji^{n-1}\right\}  \right)  \\
&  =\left\{  \varphi^{(n)}\left(  ij^{n-1}\right)  ,\varphi^{(n)}\left(
ji^{n-1}\right)  \right\}  \\
&  =\left\{  \varphi_{i}^{(n-1)}\left(  ij^{n-1}\right)  ,\varphi_{j}%
^{(n-1)}\left(  ji^{n-1}\right)  \right\}  \\
&  =\left\{  \left(  i,\varphi^{\left(  n-1\right)  }\left(  \left(  \left(
i+j\right)  (\operatorname{mod}m\right)  )^{n-1}\right)  \right)  ,\left(
j,\varphi^{\left(  n-1\right)  }\left(  \left(  \left(  j+i\right)
(\operatorname{mod}m\right)  )^{n-1}\right)  \right)  \right\}  \\
&  \in E_{\widetilde{\widetilde{S}}(n,m)}\text{ since }\left(  i+j\right)
\text{=}\left(  j+i\right)  \text{.}%
\end{align*}
Thus $\varphi^{(n)}\left(  ij^{n-1}\right)  $ and $\varphi^{(n)}\left(
ji^{n-1}\right)  $ differ in exactly one coordinate (the first).
\end{proof}
\end{theorem}

Figure 3 shows $\widetilde{\widetilde{S}}(n,3),n=1,2,3$.%
\[%
{\parbox[b]{2.9265in}{\begin{center}
\includegraphics[
trim=1.197199in 2.483964in 1.504783in 1.206214in,
height=3.6841in,
width=2.9265in
]%
{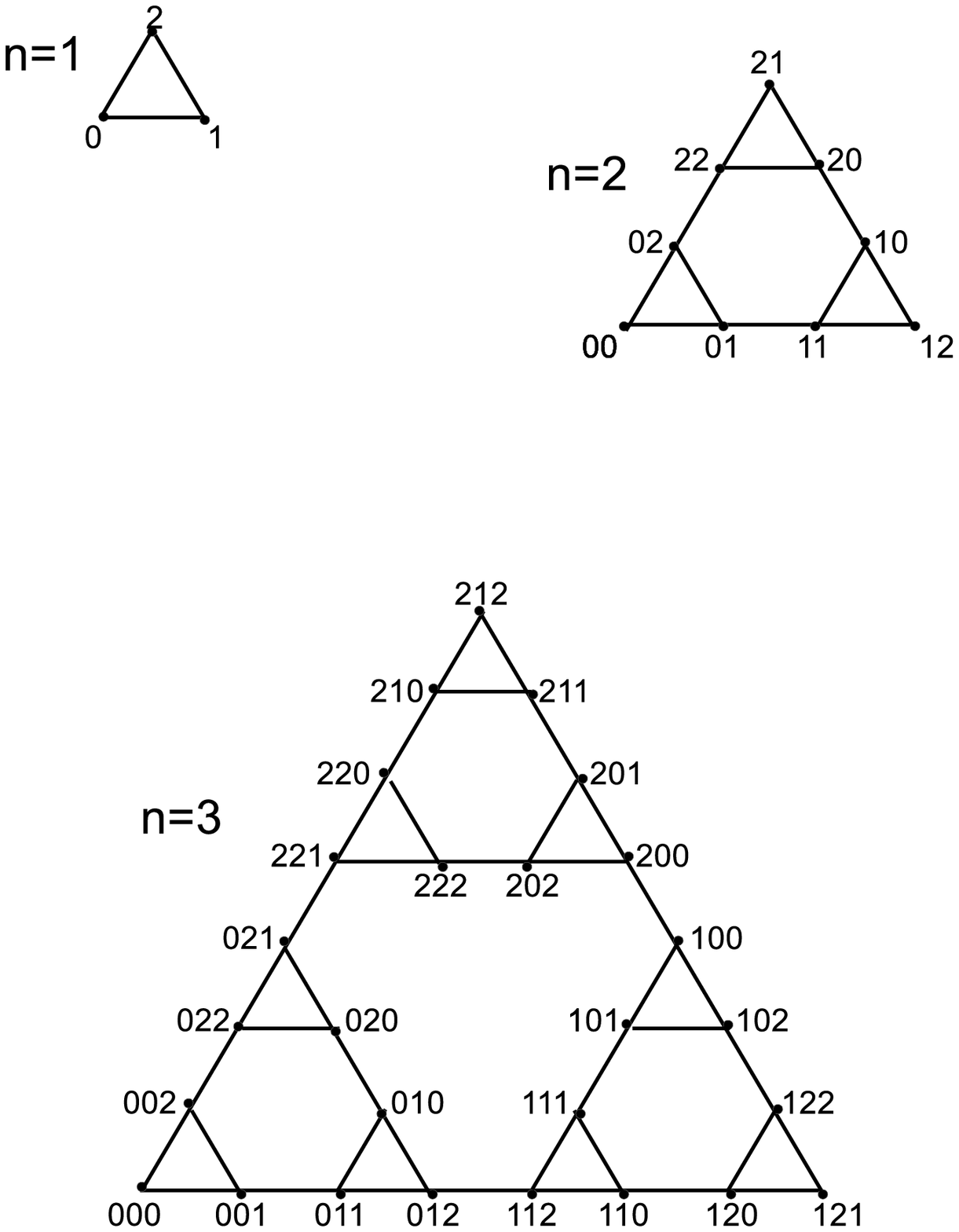}%
\\
Figure 3-$\widetilde{\widetilde{S}}(n,3)$ for $n=1,2,3$
\end{center}}}%
\]

Note that in the case $n=2$, $\varphi^{(1)}$ is the identity on $S(1,m)$ so
$\widetilde{\widetilde{S}}(2,m)=\widetilde{S}(2,m)$ and our proof includes
Lemma 1. Geometric intuition had indicated that to embed $S(n,m)$ into
$K_{m}^{n}$ we had to give the copies of $S(n-1,m)$ a twist, which led to the
definition of $\widetilde{S}(2,m)$. The counterexample showed that just one
twist was not adequate for $S(3,m)$, that we had to give the copies of
$\widetilde{S}(2,m)$ another twist, leading to the definition of
$\widetilde{\widetilde{S}}(3,m)$, \textit{etc. }

\section{Comments and Conclusions}

\subsection{Application to Layout Problems}

Theorem 2 gives the solution of a wirelength problem. In laying out a graph,
$G$, on a graph, $H$, edges of $G$ must be assigned to paths with distinct
endpoints. Since the shortest such path is a single edge, an embedding of $G$
into $H$ achieves the minimum wirelength (the sum of the lengths of all paths
assigned to edges of $G$) and bandwidth (the maximum length of any path
assigned to an edge of $G$) of any layout. Thus $WL\left(  S(n,m)\rightarrow
K_{m}^{n}\right)  =\left\vert E_{S(n,m)}\right\vert =\left(  m^{n+1}-m\right)
/2$ and $BW\left(  S(n,m)\rightarrow K_{m}^{n}\right)  =1$. In a sequel to
this paper we hope to present a solution of the bandwidth problem for laying
out $S(n,m)$ on a path, $P_{m^{n}}$, of length $m^{n}$. The problems of laying
out $S(n,m)$ on a product of paths, $P_{m^{n_{1}}}\times P_{m^{n_{2}}}%
\times...\times P_{m^{n_{k}}}$ such that $n_{1}+n_{2}+...+n_{k}=n$ so as to
minimize wirelength or bandwidth appears very challenging and will probably
require new methods.

\subsection{Relative Density of Edges}

Since $S(n,m)\simeq$ $\widetilde{\widetilde{S}}(n,m)$, a subgraph of
$K_{m}^{n}$, one may ask about the density of $E_{\widetilde{\widetilde{S}%
}(n,m)}$ in $E_{K_{m}^{n}}$:%
\begin{align*}
\frac{\left\vert E_{\widetilde{\widetilde{S}}(n,m)}\right\vert }{\left\vert
E_{K_{m}^{n}}\right\vert }  &  =\frac{\left(  m^{n+1}-m\right)  /2}{\left(
m^{n}\left(  m-1\right)  n/2\right)  }\\
&  =1/n,
\end{align*}

which is arbitrarily small as $n$ becomes large.

\subsection{In Retrospect}

Our initial guess, that $\widetilde{S}(n,m)$ is isomorphic to $S(n,m)$, proved
to be false in general, but we retained it in this presentation (for $n=2$) because

\begin{enumerate}
\item The counterexample for $n=3$ \& $m=3$ shows the subtlety of the problem, and

\item The proof for $n=2$ gave the solutions $\widetilde{\widetilde{v}}%
_{1}=v_{1}$, $\widetilde{\widetilde{v}}_{2}=\left(  v_{1}+v_{2}\right)
\operatorname{mod}m$ of the recurrence inherent in the proof of Theorem 2.
This suggested extending the formulas to $\widetilde{\widetilde{v}}%
_{k}^{\left(  n\right)  }$ as a function of $\left(  v_{1},v_{2}%
,...,v_{n}\right)  $. The resulting formulas are in the next section.
\end{enumerate}

\subsection{Formulas for $\widetilde{\widetilde{v}}_{n}$}

In the proof of Theorem 2 we have shown that
\[
\widetilde{\widetilde{v}}_{1}^{\left(  n\right)  }=\varphi_{1}^{\left(
n\right)  }\left(  v_{1},v_{2},...,v_{n}\right)  =v_{1}%
\]
and for $1<i\leq n,$ $\widetilde{\widetilde{v}}_{i}^{\left(  n\right)  }$
\begin{align*}
\widetilde{\widetilde{v}}_{i}^{\left(  n\right)  }  &  =\varphi_{i}^{\left(
n\right)  }\left(  v_{1},v_{2},...,v_{n}\right) \\
&  =\varphi_{i-1}^{\left(  n-1\right)  }\left(  \left(  v_{1}+v_{2}\right)
\operatorname{mod}m,\left(  v_{1}+v_{3}\right)  \operatorname{mod}%
m,...,\left(  v_{1}+v_{i}\right)  \operatorname{mod}m\right)  \text{.}%
\end{align*}

\begin{lemma}
$\forall n\geq i,\widetilde{\widetilde{v}}_{i}^{\left(  n\right)  }%
=\widetilde{\widetilde{v}}_{i}^{\left(  i\right)  }$

\begin{proof}
By induction on $i.$ For $i=1$, $\varphi_{1}^{\left(  n\right)  }\left(
v_{1},v_{2},...,v_{n}\right)  =v_{1}=$ $\varphi_{1}^{\left(  1\right)
}\left(  v_{1}\right)  .$ Assume the Lemma is true for $i-1\geq1$. Then if
$n>i>1$,
\begin{align}
&  \varphi_{i}^{\left(  n\right)  }\left(  v_{1},v_{2},...,v_{n}\right)
\nonumber\\
&  =\varphi_{i-1}^{\left(  n-1\right)  }\left(  \left(  v_{1}+v_{2}\right)
\operatorname{mod}m,\left(  v_{1}+v_{3}\right)  \operatorname{mod}%
m,...,\left(  v_{1}+v_{i}\right)  \operatorname{mod}m\right)  \\
&  \qquad\text{by Section 4.4,}\\
&  =\ \varphi_{i-1}^{\left(  i-1\right)  }\left(  \left(  v_{1}+v_{2}\right)
\operatorname{mod}m,\left(  v_{1}+v_{3}\right)  \operatorname{mod}%
m,...,\left(  v_{1}+v_{i}\right)  \operatorname{mod}m\right)  \nonumber\\
&  \qquad\text{by the inductive hypothesis,}\nonumber\\
&  =\ \varphi_{i}^{\left(  i\right)  }\left(  v_{1},v_{2},...,v_{i}\right)
\nonumber\\
&  \qquad\text{by Section 4.4 again.}\nonumber
\end{align}

\end{proof}
\end{lemma}

So, to express $\widetilde{\widetilde{v}}_{i}^{\left(  n\right)  }$ as a
function of $v_{1},v_{2},...,v_{n}$, we need only do it for $\widetilde
{\widetilde{v}}_{n}^{\left(  n\right)  }=\widetilde{\widetilde{v}}_{n}$.
Henceforth we shall drop the superscript, $\left(  n\right)  $, if the domain
is clear from context.

\begin{theorem}
$\widetilde{\widetilde{v}}_{n}=\varphi_{n}\left(  v_{1},v_{2},...,v_{n}%
\right)  =\left(  v_{n}+%
{\displaystyle\sum\limits_{i=1}^{n-1}}
2^{n-1-i}v_{i}\right)  \operatorname{mod}m$.

\begin{proof}
By induction on $n$. It is trivially true for $n=1$. Assume it is true for
$n\geq1$. Then%
\begin{align*}
&  \varphi_{n+1}\left(  v_{1},v_{2},...,v_{n+1}\right)  \\
&  =\varphi_{n}\left(  \left(  v_{1}+v_{2}\right)  \operatorname{mod}m,\left(
v_{1}+v_{3}\right)  \operatorname{mod}m,...,\left(  v_{1}+v_{n+1}\right)
\operatorname{mod}m\right)  \\
&  \qquad\text{by Section 4.4,}\\
&  =\left(  \left(  v_{1}+v_{n+1}\right)  +%
{\displaystyle\sum\limits_{i=1}^{n-1}}
2^{n-1-i}\left(  v_{1}+v_{i+1}\right)  \right)  \operatorname{mod}m\\
&  \qquad\text{by the inductive hypothesis,}\\
&  =\left(  v_{n+1}+\left(  v_{1}+%
{\displaystyle\sum\limits_{i=1}^{n-1}}
2^{n-1-i}v_{1}\right)  +%
{\displaystyle\sum\limits_{i=1}^{n-1}}
2^{n-1-i}v_{i+1}\right)  \operatorname{mod}m\\
&  =\left(  v_{n+1}+2^{n-1}v_{1}+%
{\displaystyle\sum\limits_{i=1}^{n-1}}
2^{\left(  n+1\right)  -1-\left(  i+1\right)  }v_{i+1}+\right)
\operatorname{mod}m\\
&  =\left(  v_{n+1}+%
{\displaystyle\sum\limits_{i=1}^{\left(  n+1\right)  -1}}
2^{\left(  n+1\right)  -1-i}v_{i}\right)  \operatorname{mod}m\text{.}%
\end{align*}

\end{proof}
\end{theorem}

\subsection{Inverting $\varphi\left(  v_{1},v_{2},...,v_{n}\right)  =\left(
\widetilde{\widetilde{v}}_{1},\widetilde{\widetilde{v}}_{2},...,\widetilde
{\widetilde{v}}_{n}\right)  $}

\begin{theorem}
$\varphi_{i}^{-1}\left(  \widetilde{\widetilde{v}}_{1},\widetilde
{\widetilde{v}}_{2},...,\widetilde{\widetilde{v}}_{n}\right)  =\left(
\widetilde{\widetilde{v}}_{i}-%
{\displaystyle\sum\limits_{j=1}^{i-1}}
\widetilde{\widetilde{v}}_{j}\right)  \operatorname{mod}m$.

\begin{proof}%
\begin{align*}
&  \left(  \varphi^{-1}\circ\varphi\right)  _{i,j}=\varphi_{i}\cdot\varphi
_{j}^{-1}\text{
\ \ \ \ \ \ \ \ \ \ \ \ \ \ \ \ \ \ \ \ \ \ \ \ \ \ \ \ \ \ \ \ \ \ \ \ \ \ \ \ \ \ \ \ \ \ \ \ \ \ \ \ }%
\\
&  =\left(  \overset{i}{\overbrace{2^{i-2},2^{i-3},...,1,1}},0,...,0\right)
\cdot\left(  \overset{j}{\overbrace{0,0,...,0,1}},-1,...,-1\right) \\
&  =\left\{
\begin{array}
[c]{ll}%
1 & \text{if }i=j\\
0 & \text{if }i\neq j
\end{array}
\right.  =\iota_{i,j}%
\end{align*}

\end{proof}
\end{theorem}

Note that this inversion is actually valid over the integers, $\mathbb{Z}$
(not just $\mathbb{Z}_{m}$), and for $n=\infty$.

When $m=2$ these formulas reduce to
\[
\widetilde{\widetilde{v}}_{n}=\left(  v_{n}+v_{n-1}\right)  \operatorname{mod}%
2
\]
and
\[
v_{n}=\left(
{\displaystyle\sum\limits_{j=1}^{n}}
\widetilde{\widetilde{v}}_{j}\right)  \operatorname{mod}2.
\]
The embedding of $S(n,2)$, a path of length $2^{n}$, into $K_{2}^{n}=Q_{n}$,
the graph of the $n$-cube, given by Theorem 2, turns out to be exactly the
Gray code mentioned in Section 3.0.1. The result is trivial for $n=1$,
$\varphi^{\left(  1\right)  }$ being the identity. If it is the Gray code for
$n-1\geq1$, then by the recurrence at the beginning of Section 4.4,
$\widetilde{\widetilde{v}}_{i}^{\left(  n\right)  }\left(  v_{1}%
,v_{2},...,v_{n}\right)  $ will be the same for the first half of $S(n,2) $ as
it was for $S(n-1,2)$ since $v_{1}=0=\widetilde{\widetilde{v}}_{1}$. For the
second half, we add $v_{1}=1$ to each component except the first, thereby
complimenting it and reversing the order. That is the (recursive) definition
of the Gray code.

\subsection{The Answer to an Old Question}

More than fifty years ago, we noticed the similarity between the natural (base
$2$) code that assigns to each $n$-tuple, $v$, of $0$s and $1$s the integer,
$\eta\left(  v\right)  $, it represents in base $2$, \textit{i.e.}
\[
\eta\left(  v\right)  =%
{\displaystyle\sum\limits_{i=1}^{n}}
2^{n-i}v_{i}\text{,}%
\]
and the Gray code (See Wikipedia). The Gray code is defined recursively as a
circular order on the vertices of the $n$-cube. The natural binary code can
also be described recursively. But can the Gray code be represented by a
formula, similar to that for $\eta\left(  v\right)  $? It seemed that some of
the coefficients in any such formula,%
\[
\gamma\left(  v\right)  =%
{\displaystyle\sum\limits_{i=1}^{n}}
c_{i}v_{i\text{,}}%
\]
would have to be negative. But then the formula would take negative, as well
as positive values, which presents a problem. We tried various ways to get
around this difficulty but without success. However, the desired formula drops
right out of our Example 3 and Section 4.5. In Example 3 we observed that
$S(n,2)$ is a path, with the vertices appearing in their natural (base 2)
order. At the end of Section 4.5 we noted that $\varphi^{\left(  n\right)
}:S(n,2)\rightarrow\widetilde{\widetilde{S}}(n,2)\subset K_{2}^{n}=Q_{n}$
carried $S(n,2)$ to the Gray code. With the inverse of $\varphi_{n}:\left\{
0,1\right\}  ^{n}\rightarrow\left\{  0,1\right\}  ^{n}$ we have our formula%
\begin{align*}
\gamma\left(  \widetilde{\widetilde{v}}\right)   &  =\eta\left(  \varphi
^{-1}\left(  \widetilde{\widetilde{v}}\right)  \right) \\
&  =%
{\displaystyle\sum\limits_{i=1}^{n}}
\left(
{\displaystyle\sum\limits_{j=1}^{i}}
\widetilde{\widetilde{v}}_{j}\left(  \operatorname{mod}2\right)  \right)
2_{\text{,}}^{n-i}\text{.}%
\end{align*}
The essence of this relationship between $\eta$ and $\gamma$ is in the
Wikipedia article "Gray Code" but presented in pseudocode rather than an
algebraic formula.

\subsection{Reverse Engineering}

In reviewing the literature on Sierpinski graphs, we were chagrined to realize
that a principle case of our title's question already had an answer: The
(generalized \& extended) Sierpinski graphs originated in work on the Tower of
Hanoi puzzle \cite{S-G-S}. A position of the Tower of Hanoi puzzle with $n$
discs may be represented by $v\in\left\{  0,1,2\right\}  ^{n}$, $v_{i}=j$
meaning that disc $i$ is on peg $j$. The position is uniquely determined by
the $n$-tuple since the discs on a peg must be stacked smaller on larger (the
radius is decreasing in $i$). When disc $i$ is moved from peg $j_{1}$ to
$j_{2}$, all discs on $j_{1}$ \& $j_{2}$ must be larger than disc $i$. Thus
the smaller discs must all be stacked on the third peg, $k\neq j_{1}$,$j_{2}$.
When such a move is made, only one coordinate of the representing $n$-tuple,
$v$, changes ($v_{i}$ goes from $j_{1}$ to $j_{2}$). So what we have done here
is reverse engineer the work of Scorer, Grundy \& Smith in that case. However,
even in that case there are interesting differences.

\subsection{The Tower of Hanoi Graph}

Let $T(n)$ be the Tower of Hanoi graph with $3$ pegs and $n$ discs. As
described in the previous paragraph, vertices represent positions and edges
represent the moving of a disc from one peg to another. Scorer, Grundy \&
Smith showed that $T(n)$ and $S(n,3)$ are different coordinatizations of the
same abstract graph. Therefore $T(n)$ and $\widetilde{\widetilde{S}}(n,3)$ are
isomorphic and both are subgraphs of $K_{3}^{n}$. So, is $T(n,3)=$
$\widetilde{\widetilde{S}}(n,3)$?

Figure 4 shows $T(n)$, $n=1,2,3$.
\[%
{\parbox[b]{2.9974in}{\begin{center}
\includegraphics[
trim=1.209095in 2.146092in 1.347592in 0.949783in,
height=3.9825in,
width=2.9974in
]%
{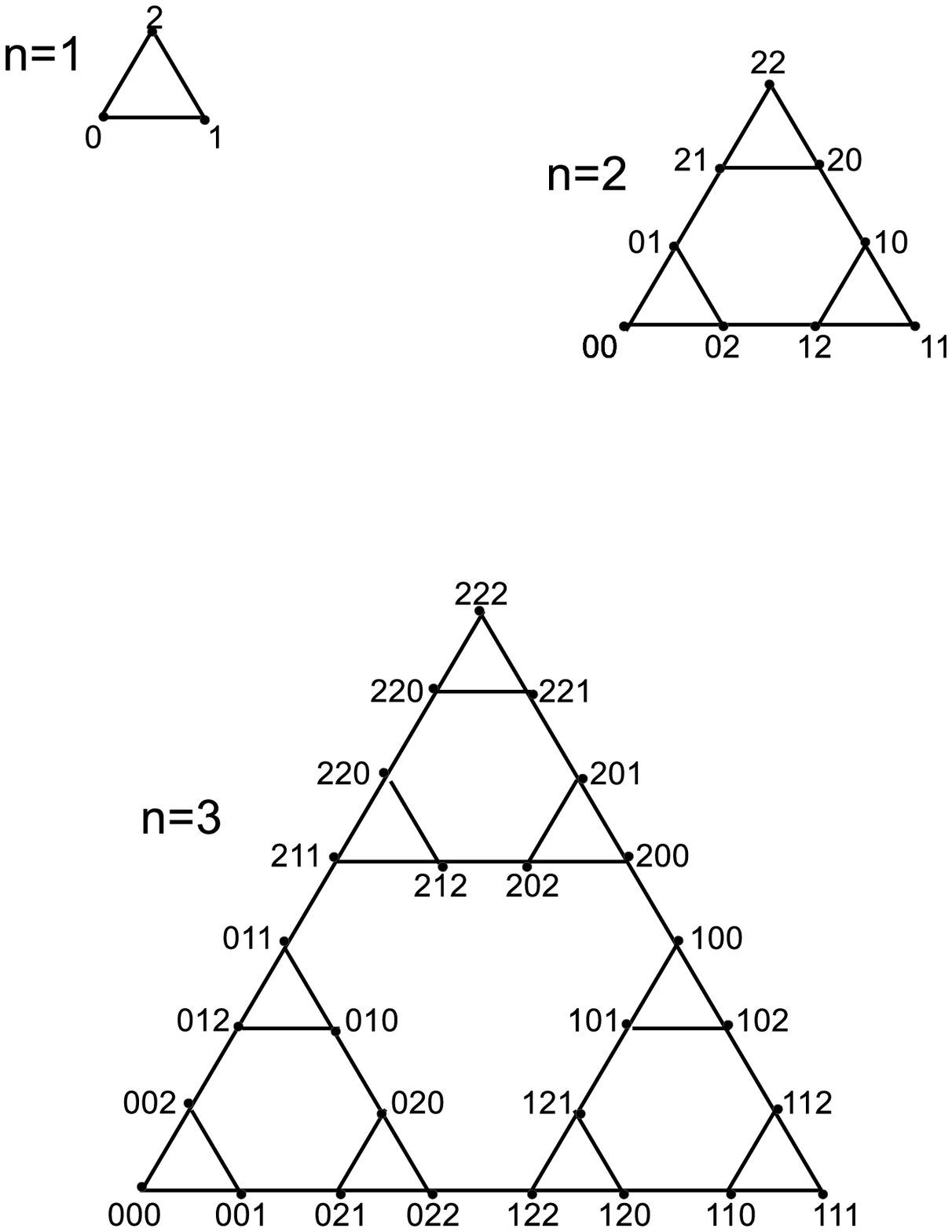}%
\\
Figure 4-$T(n)$ for $n=1,2,3$
\end{center}}}%
\]
A glance at diagrams of $\widetilde{\widetilde{S}}(n,3)$ \& $T(n)$ shows they
are not equal. Already for $n=2$ some coordinates are different. Is it
possible to change our definition of $\varphi$ a bit so that it takes $S(n,3)$
to $T(n)$?

$\tau:S(n,3)\rightarrow T(n)$ is defined recursively by the following slight
variant of $\varphi:S(n,3)\rightarrow\widetilde{\widetilde{S}}(n)$.%
\begin{align*}
\tau^{\left(  1\right)  } &  =\iota_{K_{3}}\text{, the identity on }%
K_{3}\text{,}\\
\text{For }n &  >1\text{, }\tau_{i}^{\left(  n\right)  }\left(  ij^{n-1}%
\right)  =\left(  i,\tau^{\left(  n-1\right)  }\left(  \left(  2\left(
i+j\right)  \operatorname{mod}3\right)  ^{n-1}\right)  \right)  \text{,}\\
\text{and }\tau^{\left(  n\right)  } &  =%
{\displaystyle\bigoplus\limits_{i=0}^{m-1}}
\tau_{i}^{\left(  n\right)  }\text{.}%
\end{align*}
The only difference is the factor of $2$ multiplying $\left(  i+j\right)
\operatorname{mod}3$. The diagram for $T(2)$ differs from that for
$\widetilde{\widetilde{S}}(2,3)$ by an interchange $1\longleftrightarrow2$ in
the second coordinate. This is due to the fact that when the large disc is
moved from peg $0$, the small disc must be on peg $1$ or peg $2$ and then the
large disc must go on the other peg ($2$ or $1$). Since $2\times0=0\left(
\operatorname{mod}3\right)  $, $2\times1=2\left(  \operatorname{mod}3\right)
$ and $2\times2=1\left(  \operatorname{mod}3\right)  $, multiplying $\left(
i+j\right)  $ by $2$ $\left(  \operatorname{mod}3\right)  $ exactly represents
this move.

\subsection{The Structure of $\tau$}

As with $\varphi$, $\tau$ may be viewed as a nonsingular linear transformation
$\tau\left(  v_{1},v_{2},...,v_{n}\right)  =\left(  t_{1},t_{2},...,t_{n}%
\right)  $. In particular,
\[
t_{1}=\tau_{1}\left(  v_{1},v_{2},...,v_{n}\right)  =v_{1}%
\]
and for $1<i\leq n,$
\begin{align*}
&  t_{i}=\tau_{i}^{\left(  n\right)  }\left(  v_{1},v_{2},...,v_{n}\right) \\
&  =\tau_{i-1}^{\left(  n-1\right)  }\left(  2\left(  v_{1}+v_{2}\right)
\operatorname{mod}3,2\left(  v_{1}+v_{3}\right)  \operatorname{mod}%
3,...,2\left(  v_{1}+v_{i}\right)  \operatorname{mod}3\right)  \text{.}%
\end{align*}
And, as before (Lemma 2 and then Theorem 3), we have

\begin{lemma}
$\forall n\geq i,\tau_{i}^{\left(  n\right)  }=\tau_{i}^{\left(  i\right)  }$,
\end{lemma}

which may be proved in exactly the same way. So, to express $t_{i}$ as a
function of $v_{1},v_{2},...,v_{n}$, we need only do it for $\tau_{n}^{\left(
n\right)  }=t_{n}$.

\begin{theorem}
$t_{n}=\tau_{n}^{\left(  n\right)  }\left(  v_{1},v_{2},...,v_{n}\right)
=2^{n-1}\left(  v_{n}+%
{\displaystyle\sum\limits_{i=1}^{n-1}}
2^{n-1-i}v_{i}\right)  \operatorname{mod}3$, so $t_{n}=2^{n-1}\widetilde
{\widetilde{v}}_{n}\left(  \operatorname{mod}3\right)  $.
\end{theorem}

And, also as before (Theorem 4),

\begin{theorem}
$\tau_{n}^{-1}\left(  t_{1},t_{2},...,t_{n}\right)  =\left(  2^{n-1}t_{n}%
-\sum_{j=1}^{n-1}2^{j-1}t_{j}\right)  \operatorname{mod}3$.
\end{theorem}

Again, the proofs of these theorems are essentially the same as for $\varphi$
except we need the fact that $2^{2}=1\left(  \operatorname{mod}3\right)  $ for
Theorem 6. $\tau^{-1}:T(n)\rightarrow S(n,3)$ is actually the original
Scorer-Grundy-Smith transformation.

The Tower of Hanoi is a fundamental example of a recursive algorithm. As such
it is assigned as an exercise in programming courses. The theorems above
provide a "hack" (illicit shortcut) for that exercise: In $S\left(
n,3\right)  $ the shortest path from $0^{n}$ to $1^{n}$ is $S\left(
n,2\right)  $, whose vertices appear in lexicographic order. So
\[
\tau\circ\eta^{-1}\left(  \ell\right)
\]
gives the solution of the classical Tower of Hanoi puzzle, moving discs from
$0^{n}$ to $1^{n\text{ }}$in $2^{n}-1$ steps. More specifically, the $i^{th}$
coordinate of the $\ell^{th}$ position is%
\[
\tau_{i}^{\left(  n\right)  }\left(  \ell\right)  =2^{i-1}\left(
{\displaystyle\sum\limits_{j=1}^{i-1}}
2^{i-1-j}\ell_{j}+\ell_{i}\right)  \operatorname{mod}3
\]
where
\[
\ell=\sum_{h=1}^{n}2^{n-h}\ell_{h}\text{, }\ell_{h}=0\text{ or }1\text{.}%
\]
(See, Dad, algebra can solve the Tower of Hanoi puzzle, and evidently the
Chinese Rings too!). Clifford Wolfe came up with a similar formula for
$2\tau_{i}^{\left(  n\right)  }\left(  \ell\right)  $
(http://www.clifford.at/hanoi/) by analysing the moves of individual discs:%
\[
2\tau_{i}^{\left(  n\right)  }\left(  \ell\right)  =\left(  \left(
i\operatorname{mod}2\right)  +1\right)  \left\lfloor \frac{\ell+2^{i}}%
{2^{i+1}}\right\rfloor \operatorname{mod}3\text{.}%
\]

\subsection{Generalization of the Embeddings $\varphi^{\left(  n\right)  }$,
$\tau^{\left(  n\right)  }$}

Implicit in the recursive definitions of $\varphi$ and $\tau$ are $m+1$
permutations $\pi_{i}\in\mathcal{S}_{m}$, $0\leq i\leq m$, : Let $\Pi
_{n}=\left(  \pi_{0},\pi_{1},...,\pi_{m}\right)  $ then
\[
\varepsilon_{\Pi_{n}}:S(n,m)\rightarrow K_{m}^{n}%
\]
is defined recursively by
\[
\varepsilon_{\Pi_{n}}^{\left(  1\right)  }=\iota_{K_{m}}\text{, the identity
on }K_{m},
\]
and given

$\varepsilon^{\left(  n-1\right)  }:S(n-1,m)\rightarrow K_{m}^{n-1}$,
$\varepsilon_{\Pi_{n},i}^{\left(  n\right)  }:\left(  i,S(n-1,m)\right)
\rightarrow\left(  \pi_{m}\left(  i\right)  ,K_{m}^{n-1}\right)  $ is defined
by
\begin{align*}
\varepsilon_{\Pi_{n},i}^{\left(  n\right)  }\left(  ij^{n}\right)   &
=\left(  \pi_{m}\left(  i\right)  ,\varepsilon^{\left(  n-1\right)  }\left(
\left(  \pi_{i}\left(  j\right)  \right)  ^{n}\right)  \right)  \text{ and
then}\\
\varepsilon_{\Pi_{n}}^{\left(  n\right)  }  &  =%
{\displaystyle\bigoplus\limits_{i=1}^{m}}
\varepsilon_{\Pi_{n},i}^{\left(  n\right)  }\text{.}%
\end{align*}

\begin{theorem}
$\bigskip$If $\varepsilon^{\left(  n-1\right)  }$ is an embedding and $\forall
i,j<m,$ $\pi_{i}\left(  j\right)  =\pi_{j}\left(  i\right)  $, then
$\varepsilon_{\Pi_{n}}^{\left(  n\right)  }$ is an embedding$.$
\end{theorem}

\begin{example}
For $\varphi^{\left(  n\right)  }$, $\pi_{m}=\iota_{K_{m}}$ and for $i,j<m$,
$\pi_{i}\left(  j\right)  =\left(  i+j\right)  \operatorname{mod}m$. For
$i<m$, $\pi_{i}\left(  j\right)  =\pi_{j}\left(  i\right)  $ since $i+j=j+i$.
\end{example}

\begin{example}
For $\tau^{\left(  n\right)  }$, $\pi_{m}=\iota_{K_{m}}$ and for $i,j<m$,
$\pi_{i}\left(  j\right)  =2\left(  i+j\right)  \operatorname{mod}m$ ($m$ must
be odd in order for $j\rightarrow$ $2\left(  i+j\right)  \operatorname{mod}m$
to be a permutation).
\end{example}

Note that in each of these examples $\pi_{m}=\iota_{K_{m}}$. We may always
assume, by way of normalizing to eliminate redundancy, that $\pi_{m}%
=\iota_{K_{m}}$. If $\pi_{m}\neq\iota_{K_{m}}$, the symmetry induced by
$\pi_{m}^{-1}$ will put it in that form.

Thus for each $\varepsilon^{\left(  n-1\right)  }$ we can take $\pi_{i}\left(
j\right)  $ to be $c\left(  i+j\right)  \operatorname{mod}m$, where
gcd$\left(  c,m\right)  =1$. There are $\phi\left(  m\right)  $\ ways ($\phi$
is the Euler totient function) to choose $c$, so the number of nonisomorphic
embeddings $\varepsilon^{\left(  n\right)  }:S(n,m)\rightarrow K_{m}^{n}$ we
have constructed is $\left(  \phi\left(  m\right)  \right)  ^{n}$. In the
words of Hardy \& Wright, the order of $\phi(n)$ is "always `nearly $n$'" (See
Wikepedia, "Euler's Totient Function (Growth of the Function)") so for any
$\delta>0$, for $m$ suficiently large there are at least $m^{\left(
1-\delta\right)  n}$ such embeddings.

\subsection{Constant Corners Property}

One of the striking features of $S(n,m)$ $\&$ $T(n)$ is that corners (vertices
of degree $m-1$) are constant $\left(  i^{n}\text{ for some }i\text{, }0\leq
i<m\right)  $. In $\widetilde{\widetilde{S}}(n,3)$ there is only one constant
corner, $0^{n}$, the two other corners alternating $0$ \& $1$. Also, our
mapping $\tau:S(n,3)\rightarrow T(n)$ actually preserves the coordinates of
corners ($\tau\left(  i^{n}\right)  =i^{n}$). Can $T(n)=T(n,3)$ be generalized
to $T(n,m)$, a subgraph of $K_{m}^{n}$, isomorphic to $S(n,m)$ and with
constant corners? The recursive definition of $\tau:S(n,3)\rightarrow
K_{3}^{n}$ that characterizes $T(n)$ as its range, may be extended to
$\tau:S(n,m)\rightarrow K_{m}^{n}$ for any odd $m$ by letting $c=2^{-1}\left(
\operatorname{mod}m\right)  $ in the previous section. Then%
\[
\tau_{i}^{\left(  n\right)  }\left(  ij^{n-1}\right)  =\left(  i,2^{-1}%
\tau^{\left(  n-1\right)  }\left(  \left(  \left(  i+j\right)  \left(
\operatorname{mod}m\right)  \right)  ^{n-1}\right)  \left(  \operatorname{mod}%
m\right)  \right)
\]
since multiplication by $2$ is invertible $\left(  \operatorname{mod}m\right)
$ and $2^{-1}=\frac{m+1}{2}\left(  \operatorname{mod}m\right)  $ (note that
$2^{-1}=2\left(  \operatorname{mod}3\right)  $, so it made no difference that
we used $2$ instead of $2^{-1}$ in Section 4.8, since $m=3$ there)$.$ By
induction, the $T(n,m)$ so defined has constant corners since $2^{-1}\left(
i+i\right)  =i\left(  \operatorname{mod}m\right)  $. However, the definition
does not work for even $m$ because $2$ does not have a multiplicative inverse
$\operatorname{mod}m$. $\widetilde{\widetilde{S}}\left(  n,2\right)  $
trivially has constant corners, but we can prove that there is no subgraph of
$K_{4}^{2}$ isomorphic to $S(2,4)$ with constant corners: (By contradiction)
Assume that $T(2,4)$ is such a subgraph with corners $00,11,22,33$. The
$K_{4}s$ that constitute the $K_{4}$-covering of $T(2,4)$ must lie on parallel
lines (rows or columns) of $K_{4}^{2}$. And each of those $K_{4}s$ contain
exactly one corner. If the $K_{4}s$ are rows, then the exterior edges must lie
in the columns (or \textit{vice versa}). But with one of the vertices in each
column being a corner (and not incident to an exterior edge) the $3$ remaining
vertices can only accommodate $1$ edge. Four columns then yield at most $4$
exterior edges but $T(2,4)\simeq S(2,4)$ has $\binom{4}{2}=6$ exterior edges,
a contradiction. The same argument shows that $T(2,m)$ does not exist for any
even $m>2$. The question of $T(n,m)$ for even $m>2$ and $n>2$ remains open.
Also open is the question as to whether, for odd $m$, any other
(nonisomorphic) embeddings with constant corners exist.

\subsection{Tower of Hanoi on m Pegs}

How to generalize the Tower of Hanoi puzzle to $m$ pegs, $m>3$? Obviously, if
we add more pegs and retain the same rules, the puzzle just becomes easier to
solve (although the minimum number of moves, even for $m=4$ (and arbitrary
$n$), has never been completely determined. The Frame-Stewart algorithm, which
with 4 pegs takes about $\sqrt{2n+1}$ moves, is conjectured to be optimal (see
Wikipedia, "Tower of Hanoi (Four pegs and beyond)". Scorer, Grundy \& Smith
\cite{S-G-S} presented several variants. The one we found most amusing is
called "Traveling Diplomats". The additional rules they propose for moving
discs are described in terms of arcane diplomatic protocols for moving English
diplomats from Praha to Geneva via two airlines that circulate (with flights
in both directions) through $5$ major capitols, Berlin, Praha, Rome, Geneva \&
London. So $m=5$ and the number of diplomats, $n$, is arbitrary. However,
diplomats are strictly ordered by rank and (per Scorer, Grundy \& Smith)

\begin{description}
\item[a] "Each member must always travel three stages by air, for consultation
at the two intermediate towns.

\item[b] No member may start from, visit, or end up in a town at which one of
his subordinates is stationed."
\end{description}

How was the transfer most quickly done?"

The routes of the two airlines are presented by a "map" but there is an
equivalent (and more helpful) diagram in Figure 5.
\[%
{\parbox[b]{2.2174in}{\begin{center}
\includegraphics[
trim=2.207469in 3.006730in 1.908381in 3.818943in,
height=2.1179in,
width=2.2174in
]%
{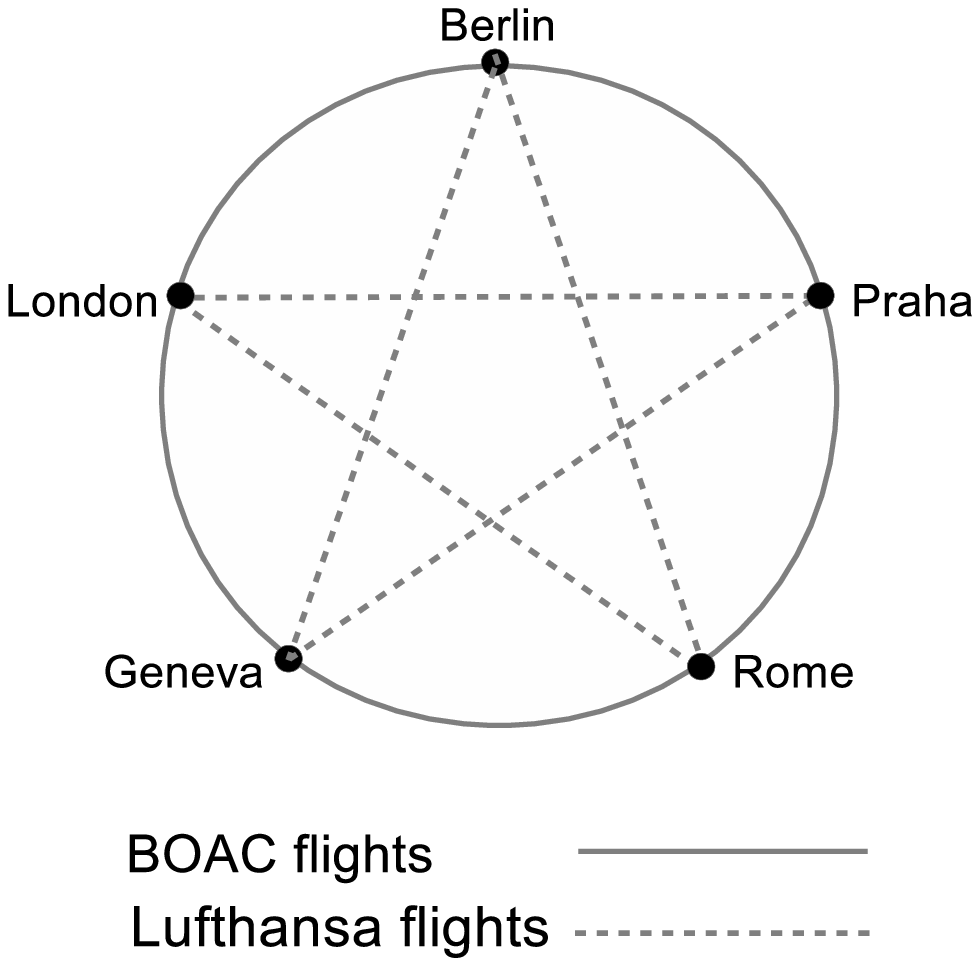}%
\\
Figure 5-Map of airline routes
\end{center}}}%
\]
With this map it is easy to see that

\begin{enumerate}
\item Any move, from town X to town Y, can be made in 3 stages on a unique airline.

\item There is just one town unvisited at each move, so by Rule $\mathbf{b}$,
all subordinates of the transient diplomat must be at that (unvisited)town.
\end{enumerate}

From this one may deduce that a move from $ik^{n-1}$ to $jk^{n-1}$, $i\neq j$
can be made iff $k=2^{-1}\left(  i+j\right)  $ which means that the graph of
the Traveling Diplomats puzzle is $T(n,5)$. Scorer, Grundy \& Smith also noted
that the rules as they state them could be extended to $T(n,p)$ for any prime
$p$.

\bigskip What other extension of the Towers of Hanoi might there be? The
conditions that seem necessary to us are

\begin{description}
\item[a] The graph should be isomorphic to $S(n,m)$,

\item[b] It should be a subgraph of $K(n,m)$, and

\item[c] It should have constant corners.
\end{description}

We have constructed such graphs, $T(n,m)$, for all $n$ if $m$ is odd, and
shown they do not exist for $n=2$ if $m>2$ is even. In our opinion, any such
$T(n,m)$ defines a generalized Towers of Hanoi puzzle: A disc may be moved
from peg $i$ iff all smaller discs are on peg $k\neq i$ and then it may be
only be moved to peg $j=\left(  2k-i\right)  \operatorname{mod}m$. It follows
then that $j\neq k$ and $k=2^{-1}\left(  i+j\right)  \operatorname{mod}m$,
making the graph of the game $T(n,m)$.

Since we have altered $\tau$ here, replacing $2\operatorname{mod}3$ by
$2^{-1}\operatorname{mod}m$ for any odd $m$, the calculation of the components
for $\tau:S(n,m)\rightarrow T(n,m)$ will look a little different:

\begin{theorem}
$t_{n}=\tau_{n}\left(  v_{1},v_{2},...,v_{n}\right)  =\left(  2^{-\left(
n-1\right)  }v_{n}+%
{\displaystyle\sum\limits_{i=1}^{n-1}}
2^{-i}v_{i}\right)  \operatorname{mod}m$, so $t_{n}=2^{n-1}\widetilde
{\widetilde{v}}_{n}\left(  \operatorname{mod}3\right)  $.
\end{theorem}

And, also as before (Theorem 4),

\begin{theorem}
$v_{n}=\tau_{n}^{-1}\left(  t_{1},t_{2},...,t_{n}\right)  =\left(
2^{n-1}t_{n}-\sum_{j=1}^{n-1}2^{j-1}t_{j}\right)  \operatorname{mod}m$.
\end{theorem}

\subsection{A New Beginning}

The Tower of Hanoi puzzle, as proposed by Eduard Lucas in 1881, started with
all $n$ discs on one peg (say $1$) and ended with all of them on another peg
(say $0$). Some puzzles, such as Rubik's Cube, are started from a random
position. If the Tower of Hanoi is started from a random position, it would
appear to be a much more complex problem.

\begin{example}
Let $n=4$ and start at $1201$. The goal is to reach $0^{4}=0000$ in the
minimum number of moves.
\end{example}

It would be simple to select a position at random by labeling the sides of a
triangular prism $0,1,2$ and tossing it $n$ times to generate a random member
of $\left\{  0,1,2\right\}  ^{n}$. In the example, starting at $1201$, the
largest disc starts on peg $1$ and must be moved to peg $0$. In order to do
that the three smaller discs must be moved to peg $2$. After the largest disc
is transferred from peg $1$ to peg $0$, the three smaller discs may be
transferred from peg $2$ to peg $0$. Thus the problem is recursively solvable
and the same holds in general. However, deciding the optimal move from a given
position, such as $1201$ seems overwhelming.

\begin{lemma}
$\forall v\in S(n,m),\exists!$ minimum length path from $v$ to $0^{n}$. Its
length is
\[
\ell\left(  v\right)  =\sum_{i=1}^{n}\overline{v}_{i}2^{n-i}%
\]
where
\[
\overline{v}_{i}=\left\{
\begin{array}
[c]{ll}%
0 & \text{if }v_{i}=0\\
1 & \text{if }v_{i}\neq0
\end{array}
\right.  \text{.}%
\]

\end{lemma}

\begin{proof}
By induction on $\ell\left(  v\right)  $. If $\ell\left(  v\right)  =0$,
$v=0^{n\text{ }}$ and the result is trivial. If true for $\ell\left(
u\right)  =\ell>0$ and $\ell\left(  v\right)  =\sum_{i=1}^{n}\overline{v}%
_{i}2^{n-i}=\ell+1$ then there are two cases:

\begin{description}
\item[Case $1$] $\overline{v}_{n}\neq0\Rightarrow v_{n}\neq0$, so replace
$v_{n}$ by $0$ to get $u\in V_{S(n,m)}$. By the definition of $S(n,m)$,
$\left\{  u,v\right\}  \in E_{S(n,m)}$ and $\sum_{i=1}^{n}\overline{u}%
_{i}2^{n-i}=\ell$. Thus the unique shortest path from $v$ to $0^{n}$ starts
with $\left\{  u,v\right\}  $ and continues with the unique shortest path from
$u$ to $0^{n}.$

\item[Case $2$] $\overline{v}_{n}=0\Rightarrow v_{n}=0\Rightarrow\exists!h<n$
such that $v_{h}\neq0$ and $v_{j}=0$ for $j>h$. Let $u\in V_{S(n,m)}$ be such
that%
\[
u_{i}=\left\{
\begin{array}
[c]{ll}%
v_{i} & \text{if }i<h\\
0 & \text{if }i=h\\
v_{h} & \text{if }i>h
\end{array}
\right.  \text{.}%
\]
Again $\left\{  u,v\right\}  \in E_{S(n,m)}$ and $\sum_{i=1}^{n}\overline
{u}_{i}2^{n-i}=\ell$ so we are done.
\end{description}
\end{proof}

Thus the antipodes of $0^{n}$ $\in V_{S(n,m)}$ are those $v$ such that
$\forall i,v_{i}\neq0$. There are $\left(  m-1\right)  ^{n}$ of them, each at
distance $2^{n}-1$ from $0^{n}$. Since $T\left(  n\right)  $ is isomorphic to
$S(n,3)$, we can use the correspondence $\tau:S(n,3)\rightarrow T(n)$ and its
inverse to solve the Tower of Hanoi problem with random starting position.

\begin{example}
For $n=4$, $\tau\left(  v_{1},v_{2},v_{3},v_{4}\right)  $ $=$%
\[
\left(  v_{1},2\left(  v_{1}+v_{2}\right)  \operatorname{mod}3,4\left(
2v_{1}+v_{2}+v_{3}\right)  \operatorname{mod}3,8\left(  4v_{1}+2v_{2}%
+v_{3}+v_{4}\right)  \operatorname{mod}3\right)
\]
in general so the matrix for $\tau$ is
\[%
\begin{array}
[c]{llll}%
1 & 0 & 0 & 0\\
2 & 2 & 0 & 0\\
2 & 1 & 1 & 0\\
2 & 1 & 2 & 2
\end{array}
\text{.}%
\]
Surprisingly, this matrix is self-inverse, so it is also the matrix for
$\tau^{-1}\left(  t_{1},t_{2},t_{3},t_{4}\right)  $. Therefore
\[
\tau^{-1}\left(  1,0,2,0\right)  =\left(
\begin{array}
[c]{llll}%
1 & 0 & 0 & 0\\
2 & 2 & 0 & 0\\
2 & 1 & 1 & 0\\
2 & 1 & 2 & 2
\end{array}
\right)  \left(
\begin{array}
[c]{l}%
1\\
0\\
2\\
0
\end{array}
\right)  =\left(
\begin{array}
[c]{l}%
1\\
2\\
1\\
0
\end{array}
\right)  \text{.}%
\]
$\ell\left(  1,2,1,0\right)  =1\cdot2^{4-1}+1\cdot2^{4-2}+1\cdot2^{4-3}%
+0\cdot2^{4-4}=14$. The vertex one step closer to $0^{4}$ in $S\left(
4,3\right)  $ is $\left(  1,2,0,1\right)  $ and the one after that is $\left(
1,2,0,0\right)  $. Continuing on in this way we generate the minimum length
path from $\left(  1,2,1,0\right)  $ to $0^{4}$ in $S\left(  4,3\right)  $.
From there we apply $\tau$ to obtain the minimum length path from $\left(
1,0,2,0\right)  $ to $0^{4}$ in $T(4,3)$:
\[%
\begin{tabular}
[c]{|r||cr}\hline
$\ell$ & $S\left(  4,3\right)  $ & \multicolumn{1}{|r|}{$T\left(  4,3\right)
$}\\\hline
$14$ & \multicolumn{1}{||r}{$1210$} & $1020$\\\hline
$13$ & \multicolumn{1}{||r}{$1201$} & $1010$\\\hline
$12$ & \multicolumn{1}{||r}{$1200$} & $1011$\\\hline
$11$ & \multicolumn{1}{||r}{$1022$} & $1211$\\\hline
$10$ & \multicolumn{1}{||r}{$1020$} & $1210$\\\hline
$9$ & \multicolumn{1}{||r}{$1002$} & $1220$\\\hline
$8$ & \multicolumn{1}{||r}{$1000$} & $1222$\\\hline
$7$ & \multicolumn{1}{||r}{$0111$} & $0222$\\\hline
$6$ & \multicolumn{1}{||r}{$0110$} & $0220$\\\hline
$5$ & \multicolumn{1}{||r}{$0101$} & $0210$\\\hline
$4$ & \multicolumn{1}{||r}{$0100$} & $0211$\\\hline
$3$ & \multicolumn{1}{||r}{$0011$} & $0010$\\\hline
$2$ & \multicolumn{1}{||r}{$0010$} & $0012$\\\hline
$1$ & \multicolumn{1}{||r}{$0001$} & $0001$\\\hline
$0$ & \multicolumn{1}{||r}{$0000$} & $0000$\\\hline
\end{tabular}
\]

\end{example}

\subsection{Another Formula}

We can also write a formula to solve the Traveling Diplomats puzzle of
\cite{S-G-S}: Suppose we take $n=4$ and number the capitols, Praha
$\rightarrow0$, Geneva $\rightarrow1$, Paris $\rightarrow2$, Rome
$\rightarrow3$, and London $\rightarrow4$ in the order of the BOAC circuit.
The challenge is to transport all four diplomats from Praha to Geneva, so we
start at $0^{4\text{ \ }}$and end at $1^{4}$ on $T(4,5).$ In $S(4,5)$ the
minimum path is $\eta^{-1}\left(  \ell\right)  $, $0\leq\ell\leq15$ and
mapping that path to $T(4,5)$ by $\tau$, whose matrix of coefficients $\left(
\operatorname{mod}5\right)  $ is
\[
\left(
\begin{array}
[c]{rrrr}%
1 & 0 & 0 & 0\\
3 & 3 & 0 & 0\\
3 & 4 & 4 & 0\\
3 & 4 & 2 & 2
\end{array}
\right)
\]
we get%
\[%
\begin{tabular}
[c]{|r||rr}\hline
$\ell$ & $S(4,5)$ & \multicolumn{1}{|r|}{$T(4,5)$}\\\hline
$0$ & $0000$ & $0000$\\\hline
$1$ & $0001$ & $0002$\\\hline
$2$ & $0010$ & $0042$\\\hline
$3$ & $0011$ & $0044$\\\hline
$4$ & $0100$ & $0344$\\\hline
$5$ & $0101$ & $0341$\\\hline
$6$ & $0110$ & $0331$\\\hline
$7$ & $0111$ & $0333$\\\hline
$8$ & $1000$ & $1333$\\\hline
$9$ & $1001$ & $1330$\\\hline
$10$ & $1010$ & $1320$\\\hline
$11$ & $1011$ & $1322$\\\hline
$12$ & $1100$ & $1122$\\\hline
$13$ & $1101$ & $1124$\\\hline
$14$ & $1110$ & $1114$\\\hline
$15$ & $1111$ & $1111$\\\hline
\end{tabular}
\]
The matrix of coefficients $\left(  \operatorname{mod}5\right)  $ for
$\tau^{-1}:$ $T(4,5)\rightarrow S(4,5)$ is
\[
\left(
\begin{array}
[c]{rrrr}%
1 & 0 & 0 & 0\\
4 & 2 & 0 & 0\\
4 & 3 & 4 & 0\\
4 & 3 & 1 & 3
\end{array}
\right)
\]
which might be used to recall a random distribution of the four diplomats
among the 5 capitols to back to Praha in the most efficient manner.

The nonzero entries of the matrices in these two examples have pseudorandom
characteristics since they are sums of exponentials reduced
$\operatorname{mod}m$ (see Wikipedia). It is amusing that they act to bring
order out of the apparent chaos of moves to solve Tower of Hanoi problems. The
key ingredient of this magic is the Scorer-Grundy-Smith transform, which takes
the combinatorial structure of $T(m,n)$ and maps it onto the geometric
structure of $S(m,n)$.

\subsection{Is $\tau:S(4,3)\rightarrow T(4,3)$ Self-inverse by Coincidence or
Principle?}

In Example 7 we noticed that $\tau=\tau^{-1}$. Is this just a coincidence or
the result of some underlying principle which make it true for all $n,m$? The
answer is somewhere in between: If $m=3$ it holds for all $n$. This can be
seen from Theorems 8 \& 9, since $2^{-1}=2=-1\left(  \operatorname{mod}%
3\right)  $. If $m=5$ it does not hold for $n=4$ (shown in Section 4.14) and
seems unlikely for any $n$ if $m>3$.

\subsection{Summing Up}

The two main things accomplished in this paper are:

\begin{enumerate}
\item Answering the question of the title in the affirmative, and

\item Making explicit what was implicit in Scorer, Grundy \& Smith's paper
\cite{S-G-S}. They described a "trilinear mapping" from $T(n)$ to $S(n,3)$ in
synthetic and qualitative terms. Utilizing the self-similarity of the
Sierpinski graph, we characterized that mapping by a recurrence and solved the
recurrence. What was synthetic and qualitative has become analytic,
computational and surprisingly efficient.
\end{enumerate}

\end{document}